\documentclass[12pt]{article}
\usepackage{amssymb,amsmath,amsthm,amssymb,multicol,
enumerate,amsfonts,euscript,array,
amsopn,hhline,lineno}
\usepackage{amssymb,amsthm,bm,mathrsfs,enumerate,array}
\usepackage{hyperref, graphtex}
\usepackage{pdflscape}
\usepackage{afterpage}
\usepackage{capt-of}
\usepackage{tikz}
\usetikzlibrary{graphs, positioning}
\newmuskip\pFqmuskip

\newcommand*\pFq[6][8]{%
	\begingroup 
	\pFqmuskip=#1mu\relax
	\mathcode`\,=\string"8000
	\begingroup\lccode`\~=`\,
	\lowercase{\endgroup\let~}\pFqcomma
	{}_{#2}F_{#3}{\left(\genfrac..{0pt}{}{#4}{#5}\bigg| #6\right)}%
	\endgroup
}
\newcommand{\pFqcomma}{\mskip\pFqmuskip}

\newtheorem{definition}{Definition}[section]
\newtheorem{theorem}[definition]{Theorem}
\newtheorem{example}[definition]{Example}
\newtheorem{lemma}[definition]{Lemma}
\newtheorem{proposition}[definition]{Proposition}
\newtheorem{remark}[definition]{Remark}
\newtheorem{corollary}[definition]{Corollary}
\newtheorem{assumption}[definition]{Assumption}

\def\C{{\mathbb C}}

\begin{document}

\title{The fundamental module of $S_3$-symmetric tridiagonal algebra associated with cycles}
\author{ John Vincent S. Morales \\ \small{Department of Mathematics and Statistics, De La Salle University}\\ \small{2401 Taft Avenue, Manila, Philippines}\\ \small{email: john.vincent.morales@dlsu.edu.ph}\\ 
}

\date{}
\maketitle

\abstract

Terwilliger \cite{terwilligerS3} recently introduced the $S_3$-symmetric tridiagonal algebra, a generalization of the tridiagonal algebra. This algebra has six generators naturally associated with the vertices of a regular hexagon: adjacent generators satisfy the tridiagonal relations, while non-adjacent ones commute. To each $Q$-polynomial distance-regular graph $\Gamma$, we associate scalars $\beta, \gamma, \gamma^*, \varrho, \varrho^*$, and define the corresponding $S_3$-symmetric tridiagonal algebra $\mathbb{T} = \mathbb{T}(\beta, \gamma, \gamma^*, \varrho, \varrho^*)$. Let $V$ denote the standard module of $\Gamma$. Then the tensor
$V^{\otimes 3} := V \otimes V \otimes V$ supports a $\mathbb{T}$-module structure, and within it exists a unique irreducible $\mathbb{T}$-submodule called the fundamental $\mathbb{T}$-module, denoted by $\Lambda$. In this paper, we focus on the case where $\Gamma$ is a cycle with vertex set $X$ and diameter $D$. We show that the associated scalars satisfy:
\begin{align*}
\beta = \zeta + \zeta^{-1}, \quad \gamma = \gamma^* = 0, \quad \varrho = \varrho^* = -(\zeta-\zeta^{-1})^2,
\end{align*}
where $\zeta$ is a fixed primitive $|X|^{\textsuperscript{th}}$ root of unity. We prove that 
\begin{align*}
\operatorname{dim}(\Lambda) & = \left\{\begin{array}{ll}
\textstyle  2D^2+2 & \text{if } |X| \text{ is even},\\
\textstyle  2D^2 + 2D +1 & \text{if } |X| \text{ is odd},
\end{array} \right.
\end{align*}
and construct two explicit bases for $\Lambda$, each of which diagonalizes half of the generators of $\mathbb{T}$. Finally, we verify that Terwilliger's conjectures \cite{terwilligerS3} hold when $\Gamma$ is a cycle.

\bigskip 



\section{Introduction}
\label{sect:introduction}

Let $S_3$ denote the symmetric group on $\{1,2,3\}$. In this paper, we study a class of algebras known as the $S_3$-symmetric tridiagonal algebras, which were introduced by Terwilliger \cite{terwilligerS3} as a natural generalization of the classical tridiagonal algebra (see Remark \ref{remk: blah}). These algebras arise in the context of $Q$-polynomial distance-regular graphs, a family of highly symmetric graphs with deep connections to algebraic combinatorics. All algebras in this paper are defined over the complex field $\mathbb{C}$.\\

Let $\beta, \gamma, \gamma^*, \varrho, \varrho^*$ denote fixed scalars. The $S_3$-symmetric tridiagonal algebra $\mathbb{T} = \mathbb{T}(\beta, \gamma, \gamma^*, \varrho, \varrho^*)$ is defined by the six elements 
\begin{align*}
\{A_1, A_2, A_3, A_1^*, A_2^*, A_3^*\}
\end{align*}
subject to the following relations.
\begin{enumerate}[(i)]
	\item (Commutativity of non-adjacent generators) For $i,j \in \{1,2,3\}$,
	\begin{align*}
	[A_i,A_j]=0, \quad [A_i^*, A_j^*]=0, \quad [A_i,A_i^*] = 0.
	\end{align*}
	\item (Tridiagonal relations for non-adjacent generators) For distinct $i,j \in \{1,2,3\}$,
	\begin{align*}
	[A_i, A_i^2A_j^* - \beta A_iA_j^*A_i + A_j^*A_i^2 - \gamma(A_iA_j^*+A_j^*A_i) - \varrho A_j^*] & = 0,\\
	[A_j^*, A_j^{*2}A_i - \beta A_j^*A_iA_j^* + A_iA_j^{*2} - \gamma^*(A_j^*A_i+A_iA_j^*) - \varrho^* A_i] & = 0. 
	\end{align*}
\end{enumerate}
Here, $[B,C] = BC - CB$ denotes the commutator.\\

We review some connections between $\mathbb{T}$ and $Q$-polynomial distance-regular graphs (see Section \ref{sect: Q-polynomial DRGs} for more information). Let $\Gamma = (X,R)$ denote a distance-regular graph with vertex set $X$, edge set $R$, distance function $\partial$, and diameter $D \geq 1$. Assume $\Gamma$ is $Q$-polynomial with eigenvalue sequence $\{\theta_i\}_{i=0}^D$ and dual eigenvalue sequence $\{\theta_i^* \}_{i=0}^D$. By construction, $\theta_i, \theta_i^*$ are real. By \cite[Lemma 3.5]{terwilligerDRG}, the scalars $\{\theta_i \}_{i=0}^D$ are mutually distinct. By \cite[Lemma 11.7]{terwilligerDRG}, the scalars $\{\theta_i^* \}_{i=0}^D$ are mutually distinct. Using the eigenvalue sequences of $\Gamma$, we obtain scalars $\beta, \gamma, \gamma^*, \varrho, \varrho^*$ and define the corresponding $S_3$-symmetric tridiagonal algebra $\mathbb{T} = \mathbb{T}(\beta, \gamma, \gamma^*, \varrho, \varrho^*)$.\\

Let $V$ denote the vector space over $\mathbb{C}$ with orthonormal basis $X$. We call $V$ the standard module of $\Gamma$. By \cite[Theorem 5.4]{terwilligerS3}, the tensor space $V^{\otimes 3} := V \otimes V \otimes V$ admits a $\mathbb{T}$-module structure. By \cite[Proposition 9.6 and Definition 9.7]{terwilligerS3}, there exists a unique irreducible $\mathbb{T}$-submodule of $V^{\otimes 3}$ containing the vector $\mathbf{1}^{\otimes 3} := \sum_{x,y,z \in X} x \otimes y \otimes z$. This irreducible $\mathbb{T}$-module is called the fundamental $\mathbb{T}$-module associated with $\Gamma$ and is denoted by $\Lambda$.\\

In this paper, our main focus is the case where $\Gamma$ is a cycle graph. We investigate the structure and dimension of the associated fundamental $\mathbb{T}$-module $\Lambda$. We provide $\Lambda$ with two explicit bases where one basis diagonalizes half of the six generators of $\mathbb{T}$ while the other basis diagonalizes the remaining generators. We also describe the matrix of transition from one basis to another. Finally, we confirm that the conjectures proposed by Terwilliger \cite[Section 12]{terwilligerS3} hold true for cycles. 

\begin{remark}\label{remk: blah}\rm 
	The tridiagonal algebra is as follows. For scalars $\beta, \gamma, \gamma^*, \varrho, \varrho^*$ the tridiagonal algebra $T=T(\beta,\gamma, \gamma^*, \varrho, \varrho^*)$ is defined by generators $A,A^*$ and relations 
	\begin{align}
	[A, A^2A^* - \beta AA^*A + A^*A^2 - \gamma(AA^*+A^*A) - \varrho A^*] & = 0, \label{eqn: tdalgebrarels1}\\
	[A^*, A^{*2}A - \beta A^*AA^* + AA^{*2} - \gamma^*(A^*A+AA^*) - \varrho^* A] & = 0. \label{eqn: tdalgebrarels2}
	\end{align}
	We call \eqref{eqn: tdalgebrarels1}--\eqref{eqn: tdalgebrarels2} the tridiagonal relations which first appeared in \cite[Lemma 5.4]{terwilligerSUBCONS}. Some special cases of the tridiagonal
	algebra are (i) the enveloping algebra of
	the Onsager Lie algebra (see \cite[Example 3.2 and Remark 3.8]{terwilligerOnsagerLiealgebra}) where 
	\begin{align*}
	\beta = 2, \qquad \gamma = \gamma^* = 0, \qquad \varrho \neq 0, \qquad \varrho^* \neq 0;
	\end{align*}
	(ii) the positive part of the $q$-deformed
	enveloping algebra $U_q(\widehat{\mathfrak{sl}}_2)$ (see \cite[Example 1.7 and Remark 10.2]{itotanabeterwilliger}) where 
	\begin{align*}
	\beta = q^2 + q^{-2} \neq \pm 2, \quad \gamma = \gamma^* = 0, \quad \varrho = \varrho^* = 0;
	\end{align*}
	and (iii) the $q$-Onsager algebra (see \cite[Section 2]{baseilhacA}, \cite[Section 1]{baseilhacB}, and \cite[Section 1.2]{itoterwilliger}) where
	\begin{align*}
	\beta = q^2 + q^{-2} \neq \pm 2, \quad \gamma = \gamma^* = 0, \quad \varrho = \varrho^* = -(q^2 - q^{-2})^2.
	\end{align*}
\end{remark}

\section{$Q$-polynomial distance-regular graphs}
\label{sect: Q-polynomial DRGs}

Here, we review concepts on $Q$-polynomial distance-regular graphs. For more information, see \cite{Bannai,Brouwer,damkoolentanaka,Godsil,martintanaka,terwilligerDRG, terwilligerSUBCONS1}.\\


For a nonempty finite set $X$, let $V$ denote the vector space over $\mathbb{C}$ with orthonormal basis $X$ and Hermitian inner product $\langle u,v \rangle = \bar{u}^tv$ for all $u,v \in V$ where $\ \bar{ }\ $ and $^{t}$ denotes complex conjugate and transpose, respectively. 
Let $\Gamma = (X,R)$ be a finite, undirected, simple connected graph with vertex set $X$ and edge set $R$. We call $V$ the standard module of $\Gamma$. Let $\operatorname{End}(V)$ denote the algebra of all linear maps $V \rightarrow V$ and identify it with the vector space over $\mathbb{C}$ with basis $X \times X = \{xy\ |\ x,y \in X \}$. Let $\partial$ denote the path-length distance function on $\Gamma$. Define the diameter $D = \operatorname{max}\{\partial(a,b)\ |\ a,b \in X \}$. 

\subsection{Distance-regularity}

The graph $\Gamma$ is called distance-regular if for all integers $h,i,j \in \{0,1,\ldots,D\}$ the number
\begin{align}\label{eqn: the scalars p h i j}
p_{ij}^h:=\left|\left\{z \in X\ |\ \partial(x,z)=i \text{ and } \partial(z,y)=j \right\} \right|
\end{align}
is independent of the choice of $x,y \in X$ such that $\partial(x,y)=h$. The scalars $p^h_{ij}$ are known as intersection numbers of $\Gamma$. By triangle inequality, the following hold for $h,i,j \in \{0,1,\ldots,D\}$:
\begin{enumerate}
	\item[(A1)] $p^{h}_{ij} = 0$ if one of $h,i,j$ is greater than the sum of the other two,
	\item[(A2)] $p^{h}_{ij} \neq 0$ if one of $h,i,j$ is equal to the sum of the other two.
\end{enumerate}

\subsection{Bose--Mesner algebra of $\Gamma$}

From here on, we assume $\Gamma$ is distance-regular with diameter $D \geq 1$. Define matrices $A_0, A_1, \ldots, A_D$ in $\operatorname{End}(V)$ by
\begin{align}
A_i & = \textstyle \sum_{x,y \in X} \delta_{i,\partial(x,y)}\ xy, \quad (i \in \{0,1,\ldots,D\}),
\end{align}
where $\delta$ denotes Kronecker delta. These matrices are called the distance matrices of $\Gamma$ and satisfy:
\begin{enumerate}
	\item[(B1)] $A_0 = I$ where $I= \textstyle \sum_{x,y \in X} \delta_{x,y}\ xy$,
	\item[(B2)] $A_0 + A_1 + \cdots + A_{D} = J$ where $J = \textstyle \sum_{x,y \in X} xy$,
	\item[(B3)] $A_i^t = A_i$ for $i \in \{0,1,\ldots,D\}$,
	\item[(B4)] $\bar{A}_i = A_i$ for $i \in \{0,1,\ldots,D\}$,
	\item[(B5)] $A_iA_j = \sum_{h=0}^{D} p_{ij}^h A_h$ for $i,j \in \{0,1,\ldots,D\}$.
\end{enumerate}
The span of $A_0,A_1,\ldots,A_D$ forms a commutative subalgebra of $\operatorname{End}(V)$ called the Bose--Mesner algebra of $\Gamma$, denoted by $M$. It has a second basis $E_0, E_1, \ldots, E_D$ called the primitive idempotents of $\Gamma$ which satisfy: 
\begin{enumerate}
	\item[(C1)] $E_0 = |X|^{-1}J$,
	\item[(C2)] $E_0 + E_1 + \cdots + E_D = I$,
	\item[(C3)] $\bar{E}_i = E_i$ for $i \in \{0,1,\ldots,D\}$,
	\item[(C4)] $E_i^t = E_i$ for $i \in \{0,1,\ldots,D\}$,
	\item[(C5)] $E_iE_j = \delta_{i,j}E_i$ for $i,j \in \{0,1,\ldots,D\}$.
\end{enumerate}

\subsection{The eigenvalues of $\Gamma$ and the $Q$-polynomial property}

Since $\{E_i\}_{i=0}^D$ is a basis for $M$, there exist scalars $\{\theta_i\}_{i=0}^D$ such that $A_1 = \sum_{i=0}^D \theta_i E_i$. By (C5), $A_1E_i=E_iA_1=\theta_i E_i$ for $i \in \{0,1,\ldots,D\}$. The scalars $\{\theta_i\}_{i=0}^D$ are real and mutually distinct. We call $\{\theta_i\}_{i=0}^D$ the eigenvalue sequence of $\Gamma$.\\

Let $\circ$ denote entry-wise multiplication in $\operatorname{End}(V)$. By (B2), we have $A_i \circ A_j = \delta_{i,j} A_i$ for $i,j \in \{0,1,\ldots,D\}$ and so $M$ is closed under $\circ$. Thus, there exist scalars $q^h_{ij}$ where
\begin{align}\label{eqn: the scalars q h i j}
E_i \circ E_j & = \textstyle |X|^{-1} \sum_{h=0}^D q^{h}_{ij} E_h \qquad (i,j \in \{0,1,\ldots,D\}). 
\end{align}
The $q^h_{ij}$ are real and nonnegative for $h,i,j \in \{0,1,\ldots,D\}$. We say $\Gamma$ is $Q$-polynomial with respect to the ordering $E_0, E_1, \ldots, E_D$ whenever for all distinct $h,j \in \{0,1,\ldots,D\}$, $q^{h}_{1j} = 0$ if and only if $|h-j| \neq 1$.\\

For the rest  of the section, we assume $\Gamma$ is $Q$-polynomial with respect to the ordering $E_0, E_1, \ldots, E_D$. Then the following hold for $h,i,j \in \{0,1,\ldots,D\}$:
\begin{enumerate}
	\item[(D1)] $q^{h}_{0j} = \delta_{hj}$,
	\item[(D2)] $q^{h}_{i0} = \delta_{hi}$,
	\item[(D3)] $q^{0}_{ij} \neq 0$ if and only if $i = j$.
\end{enumerate}

\subsection{Dual Bose--Mesner algebra of $\Gamma$}

Fix a vertex $x \in X$. For $i \in \{0,1,\ldots,D\}$, define the diagonal matrix $E_i^* := E_i^*(x)$ in $\operatorname{End}(V)$ by
\begin{align}
E_i^* & = \textstyle \sum_{y \in X} \delta_{i, \partial(x,y)}\ yy 
\end{align}
We call $\{E_i^*\}_{i=0}^D$ the dual primitive idempotents of $\Gamma$ with respect to $x$. Observe that 
\begin{enumerate}
	\item[(E1)] $E_0^* + E_1^* + \cdots +E_D^* = I$,
	\item[(E2)] $\bar{E}_i^* = E_i^*$ for $i \in \{0,1,\ldots,D\}$,
	\item[(E3)] $E_i^{*t} = E_i^*$ for $i \in \{0,1,\ldots,D\}$,
	\item[(E4)] $E_i^*E_j^* = \delta_{i,j}E_j^*$ for $i,j \in \{0,1,\ldots,D\}$.
\end{enumerate}
The span of $E_0^*,E_1^*,\ldots,E_D^*$ forms a commutative subalgebra of $\operatorname{End}(V)$ canned the dual Bose--Mesner algebra of $\Gamma$ with respect to $x$, denoted by $M^*:=M^*(x)$.\\ 

For $i \in \{0,1,\ldots,D\}$, define the diagonal matrix $A_i^* := A_i^*(x)$ in $\operatorname{End}(V)$ such that the $yy$-entry of $A_i^*$ is given by
\begin{align}\label{eqn: dual distance matrices of Gamma}
(A_i^*)_{yy} & = |X|(E_i)_{xy} \qquad (y \in X),
\end{align}
where $(E_i)_{xy}$ means $xy$-entry of $E_i$. The matrices $A_0^*, A_1^*, \ldots, A_D^*$ form a second basis for $M^*$ and satisfy:
\begin{enumerate}
	\item[(F1)] $A_0^* = I$,
	\item[(F2)] $A_0^* + A_1^* + \cdots + A_D^* = |X|E_0^*$,
	\item[(F3)] $A_i^{*t} = A_i^*$ for $i \in \{0,1,\ldots,D\}$,
	\item[(F4)] $A_i^*A_j^* = \textstyle \sum_{h=0}^D q^{h}_{ij} A_h^*$ for $i,j \in \{0,1,\ldots,D\}$. 
\end{enumerate}
We call $\{A_i^*\}_{i=0}^D$ the dual distance matrices of $\Gamma$ with respect to $x$.

\subsection{Dual eigenvalues of $\Gamma$ and the scalars $\beta, \gamma, \gamma^*, \varrho, \varrho^*$}

Since $\{E_i^*\}_{i=0}^D$ is a basis of $M^*$, there exist scalars $\{\theta_i^*\}_{i=0}^D$ such that $A_1^* = \sum_{i=0}^D \theta_i^* E_i^*$. By (E4), $A_1^*E_i^* = E_i^*A_1 = \theta_i^*E_i^*$ for $i \in \{0,1,\ldots,D\}$. The scalars $\{\theta_i^*\}_{i=0}^D$ are real and mutually distinct. We call $\{\theta_i^*\}_{i=0}^D$ the dual eigenvalue sequence of $\Gamma$ with respect to $x$. 

\begin{lemma}\label{lem: the scalars beta, gamma, etc. from Q-poly DRG} \cite[Lemma 5.4]{terwilligerSUBCONS}
	Assume $\Gamma$ is a $Q$-polynomial distance-regular graph with eigenvalue sequence $\{\theta_i\}_{i=0}^D$ and dual eigenvalue sequence $\{\theta_i^* \}_{i=0}^D$. Then there exist scalars $\beta, \gamma, \gamma^*, \varrho, \varrho^*$ associated with $\Gamma$ satisfying (i)--(iii) below. 
	\begin{enumerate}[(i)]
		\item $\beta + 1$ is equal to each of 
		\begin{align*}
		\frac{\theta_{i-2}-\theta_{i+1}}{\theta_{i-1}-\theta_i}, \qquad \frac{\theta_{i-2}^*-\theta_{i+1}^*}{\theta_{i-1}^*-\theta_i^*},
		\end{align*}
		for $2 \leq i \leq D-1$. 
		\item For $1 \leq i \leq D-1$, both
		\begin{align*}
		\gamma = \theta_{i-1}-\beta \theta_i + \theta_{i+1}, \quad \gamma^* = \theta_{i-1}^*-\beta \theta_i^* + \theta_{i+1}^*.
		\end{align*}
		\item For $1 \leq i \leq D$, both
		\begin{align*}
		\varrho & = \theta_{i-1}^2 - \beta \theta_{i-1}\theta_i + \theta_i^2 - \gamma(\theta_{i-1}+\theta_i), \\
		\varrho^* & = \theta_{i-1}^{*2} - \beta \theta_{i-1}^*\theta_i^* + \theta_i^{*2} - \gamma^*(\theta_{i-1}^*+\theta_i^*).
		\end{align*}
	\end{enumerate}
\end{lemma}

\subsection{Terwilliger algebra of $\Gamma$ with respect to $x$}\label{subsection: Terwilliger algebra of Gamma}

	Let $\Gamma = (X,R)$ be a $Q$-polynomial distance-regular graph with standard module $V$. By Lemma \ref{lem: the scalars beta, gamma, etc. from Q-poly DRG}, there exist some scalars $\beta, \gamma, \gamma^*, \varrho, \varrho^*$ associated with $\Gamma$. Consider the tridiagonal algebra $T=T(\beta, \gamma, \gamma^*, \varrho, \varrho^*)$ (see Remark \ref{remk: blah}). Then $V$ admits a $T$-module structure where $A$ acts as the matrix $A_1$ and $A^*$ acts as the matrix $A_1^*$. In other words, there exists an algebra homomorphism $T \rightarrow \operatorname{End}(V)$ such that $A \mapsto A_1$ and $A^* \mapsto A_1^*$. The image of this map is called the Terwilliger algebra of $\Gamma$ with respect to $x$ (see \cite[Section 11.2]{terwilligerDRG} and \cite[Definition 3.3]{terwilligerSUBCONS1}).

\section{The tensor space $V^{\otimes 3}$ as a $\mathbb{T}$-module}\label{sect: linear maps on tensor space}

We continue to discuss the $Q$-polynomial distance-regular graphs and adopt the notations from Section \ref{sect: Q-polynomial DRGs}. For $x \in X$, let $\Gamma(x) = \{y \in X\ |\ \partial(x,y) = 1 \}$. Note that 
\begin{align}\label{eqn: action of A1 on each basis vector}
A_1x = \textstyle \sum_{\xi \in \Gamma(x)} \xi, \quad (x \in X).
\end{align}

\begin{definition}\label{defn: the adjacency maps on tensor V}\rm \cite[Definition 6.2 and Lemma 6.3]{terwilligerS3}
	For $r \in \{1,2,3\}$, define the linear transformation $A^{(r)}:V^{\otimes 3}\rightarrow V^{\otimes 3}$ such that 
	$$\begin{array}{ccc}
	A^{(1)} & \text{ acts as } & A_1 \otimes I \otimes I,\\
	A^{(2)} & \text{ acts as } & I \otimes A_1 \otimes I,\\
	A^{(3)} & \text{ acts as } & I \otimes I \otimes A_1.
	\end{array} $$
\end{definition}

\begin{definition}\label{defn: the primitive idempotent maps on tensor V}\rm \cite[Definition 6.5 and Lemma 6.6]{terwilligerS3} 
	For each $r \in \{1,2,3\}$ and for $i \in \{0,1,\ldots,D\}$, define the linear map $E_i^{(r)}:V^{\otimes 3} \rightarrow V^{\otimes 3}$ such that 
	\begin{align*}
	E_i^{(1)}(u \otimes v \otimes w) & = E_iu \otimes v \otimes w \quad (u,v,w \in V),\\
	E_i^{(2)}(u \otimes v \otimes w) & = u \otimes E_iv \otimes w \quad (u,v,w \in V),\\
	E_i^{(3)}(u \otimes v \otimes w) & = u \otimes v \otimes E_iw \quad (u,v,w \in V),
	\end{align*}
	where $E_0, E_1, \ldots, E_D$ are the primitive idempotents of $\Gamma$. The map $E_i^{(1)}, E_i^{(2)}, E_i^{(3)}$ are the primitive idempotents of $A^{(1)}, A^{(2)}, A^{(3)}$, respectively.
\end{definition}

Let $r\in \{1,2,3\}$ be given. Since $A_1 = \sum_{i=0}^D \theta_iE_i$ and by (C5), we have
\begin{enumerate}
	\item[(G1)] $A^{(r)} = \sum_{i=0}^D \theta_i E_i^{(r)}$,
	\item[(G2)] $E_i^{(r)}E_j^{(r)} = \delta_{i,j}E_i^{(r)}$ for $i,j \in \{0,1,\ldots,D\}$,
	\item[(G3)] $A^{(r)}E_i^{(r)} = \theta_i E_i^{(r)} = E_i^{(r)}A^{(r)}$ for $i \in \{0,1,\ldots,D\}$
	\item[(G4)] $I^{\otimes 3} := I \otimes I \otimes I = \sum_{i=0}^D E_i^{(r)}$,
	\item[(G5)] $E_i^{(r)} = \prod_{\substack{0 \leq j \leq D, j \neq i}} \frac{A^{(r)}-\theta_jI}{\theta_i - \theta_j}$.
\end{enumerate}

Since $X$ is an orthonormal basis of $V$, it follows that
\begin{align}\label{eqn: orthonormal basis of tensor V}
\{x \otimes y \otimes z\ |\ x,y,z \in X \}
\end{align}
is a basis of $V^{\otimes 3}$. Thus, we can define a linear map $V^{\otimes 3} \rightarrow V^{\otimes 3}$ by its action on \eqref{eqn: orthonormal basis of tensor V}.

\begin{definition}\label{defn: the dual adjacency maps on tensor V}\rm \cite[Definition 7.1 and Lemma 7.2]{terwilligerS3}
	For $r \in \{1,2,3\}$, define the linear map $A^{*(r)}:V^{\otimes 3} \rightarrow V^{\otimes 3}$ as follows. For $x,y,z \in X$, we have
	\begin{align*}
	A^{*(1)}(x \otimes y \otimes z) & = x \otimes y \otimes z\ \theta_{\partial(y,z)}^*,\\
	A^{*(2)}(x \otimes y \otimes z) & = x \otimes y \otimes z\ \theta_{\partial(x,z)}^*,\\
	A^{*(3)}(x \otimes y \otimes z) & = x \otimes y \otimes z\ \theta_{\partial(x,y)}^*,
	\end{align*}
	where $\{\theta_i^*\}_{i=0}^{D}$ is the dual eigenvalue sequence of $\Gamma$. Note that each of $A^{*(1)}, A^{*(2)}, A^{*(3)}$ is diagonalizable since \eqref{eqn: orthonormal basis of tensor V} is a basis of $V^{\otimes 3}$. 
\end{definition}

\begin{definition}\label{defn: the dual primitive idempotent maps on tensor V}\rm \cite[Definition 7.3 and Lemma 7.4]{terwilligerS3}
	For each $r \in \{1,2,3\}$ and for $i \in \{0,1,\ldots,D\}$, define the linear map $E_i^{*(r)}:V^{\otimes 3} \rightarrow V^{\otimes 3}$ as follows. For $x,y,z \in X$, 
	\begin{align*}
	E_i^{*(1)}(x \otimes y \otimes z) & = x \otimes y \otimes z\ \delta_{i,\partial(y,z)}.\\
	E_i^{*(2)}(x \otimes y \otimes z) & = x \otimes y \otimes z\ \delta_{i,\partial(x,z)}.\\
	E_i^{*(3)}(x \otimes y \otimes z) & = x \otimes y \otimes z\ \delta_{i,\partial(x,y)}.
	\end{align*}
	The maps $E_i^{*(1)}, E_i^{*(2)}, E_i^{*(3)}$ are the primitive idempotents of $A^{*(1)}, A^{*(2)}, A^{*(3)}$, respectively.
\end{definition}

	Let $r \in \{1,2,3\}$ be given. By Definitions \ref{defn: the dual adjacency maps on tensor V}--\ref{defn: the dual primitive idempotent maps on tensor V}, we have
	
	\begin{enumerate}
	\item[(H1)] $A^{*(r)} = \sum_{i=0}^D \theta_i^* E_i^{*(r)}$,
	\item[(H2)] $E_i^{*(r)}E_j^{*(r)} = \delta_{i,j}E_i^{*(r)}$ for $i,j \in \{0,1,\ldots,D\}$,
	\item[(H3)] $A^{*(r)}E_i^{*(r)} = \theta_i^* E_i^{*(r)} = E_i^{*(r)}A^{*(r)}$ for $i \in \{0,1,\ldots,D\}$
	\item[(H4)] $I^{\otimes 3} = \sum_{i=0}^D E_i^{*(r)}$,
	\item[(H5)] $E_i^{*(r)} = \prod_{\substack{0 \leq j \leq D, j \neq i}} \frac{A^{*(r)}-\theta_j^*I}{\theta_i^* - \theta_j^*}$.
	\end{enumerate}

\begin{theorem}\cite[Theorem 5.4, Corollary 8.7]{terwilligerS3}\label{thm: S3-symmetric tridiagonal algebra acting on tensor power} Let $\Gamma$ denote a $Q$-polynomial distance-regular graph with eigenvalue sequence $\{\theta_i\}_{i=0}^D$ and dual eigenvalue sequence $\{\theta_i^* \}_{i=0}^D$. Let $\beta, \gamma, \gamma^*, \varrho, \varrho^*$ be the scalars from Lemma \ref{lem: the scalars beta, gamma, etc. from Q-poly DRG} and consider the $S_3$-symmetric tridiagonal algebra $\mathbb{T} = \mathbb{T}(\beta, \gamma, \gamma^*, \varrho, \varrho^*)$. Then $V^{\otimes 3}$ admits a $\mathbb{T}$-module structure where, for $r \in \{1,2,3\}$, the generators $A_r$ and $A_r^*$ of $\mathbb{T}$ act as $A^{(r)}$ and $A^{*(r)}$, respectively.
\end{theorem}

\begin{proposition}\label{prop: fundamental module Lambda}\cite[Proposition 9.6]{terwilligerS3}
	There exists a unique irreducible $\mathbb{T}$-submodule $\Lambda$ of $V^{\otimes 3}$ that contains $\mathbf{1}^{\otimes 3} := \sum_{x,y,z \in X} x \otimes y \otimes z$. 
\end{proposition}

We refer to $\Lambda$ as the fundamental $\mathbb{T}$-module (see \cite[Definition 9.7]{terwilligerS3}). We end this section with some vectors found in $\Lambda$.

\begin{definition}\rm \cite[Definition 9.9]{terwilligerS3}\label{defn: the vectors P h i j}
	For $h,i,j \in \{0,1,\ldots, D\}$, define
	\begin{align*}
	P_{h,i,j} & = \textstyle \sum_{\substack{x, y, z \in X\\ \partial(y,z)=h\\ \partial(x,z) = i\\ \partial(x,y)=j}} x \otimes y \otimes z.
	\end{align*}
\end{definition}

By \cite[Lemma 9.10, Lemma 9.11, and Lemma 9.12]{terwilligerS3}, we have
\begin{enumerate}
	\item[(I1)] For $h,i,j \in \{0,1,\ldots,D\}$, $P_{h,i,j} = E_h^{*(1)}E_i^{*(2)}E_j^{*(3)}(\mathbf{1}^{\otimes 3})$,
	\item[(I2)] For $h,i,j \in \{0,1,\ldots,D\}$, $P_{h,i,j} \in \Lambda$,
	\item[(I3)] The vectors in $\{P_{h,i,j}\ |\ h,i,j\in \{0,1,\ldots,D\} \}$ are mutually orthogonal,
	\item[(I4)] For $h,i,j\in \{0,1,\ldots,D\}$, $P_{h,i,j} = \mathbf{0}$ if and only if $p^{h}_{i,j} = 0$. 
\end{enumerate}

\begin{definition}\label{defn: the vectors Q h i j}\rm \cite[Definition 9.14]{terwilligerS3} For $h,i,j \in \{0,1,\ldots,D\}$, define
	\begin{align*}
	Q_{h,i,j} & = \textstyle |X| \sum_{x \in X} E_hx \otimes E_ix \otimes E_jx.
	\end{align*}
\end{definition}

By \cite[Lemma 9.15, Lemma 9.16, and Lemma 9.17]{terwilligerS3}, we have
\begin{enumerate}
	\item[(J1)] For $h,i,j \in \{0,1,\ldots,D\}$, $Q_{h,i,j} = |X|E_h^{(1)}E_i^{(2)}E_j^{(3)}(P_{0,0,0})$,
	\item[(J2)] For $h,i,j \in \{0,1,\ldots,D\}$, $Q_{h,i,j} \in \Lambda$,
	\item[(J3)] The vectors in $\{Q_{h,i,j}\ |\ h,i,j\in \{0,1,\ldots,D\} \}$ are mutually orthogonal,
	\item[(J4)] For $h,i,j\in \{0,1,\ldots,D\}$, $Q_{h,i,j} = \mathbf{0}$ if and only if $q^{h}_{ij} = 0$. 
\end{enumerate}

\section{The tensor space $V^{\otimes 3}$ as a module for $\operatorname{Aut}(\Gamma)$}

We continue with the $Q$-polynomial distance-regular graph $\Gamma$ and use notations from Sections \ref{sect: Q-polynomial DRGs}--\ref{sect: linear maps on tensor space}. In this section, we explore how $V^{\otimes 3}$ is a module for the automorphism group of $\Gamma$.\\

By an automorphism of $\Gamma$, we mean a permutation $g: X \rightarrow X$ satisfying $\partial(x,y) = \partial(g(x),g(y))$ for all $x,y \in X$. The set $\operatorname{Aut}(\Gamma)$ of all automorphisms of $\Gamma$ forms a group under composition. Throughout, let $G$ denote a subgroup of $\operatorname{Aut}(\Gamma)$.

\begin{definition}\label{defn: G action on X times X times X}\rm 
	We endow $X \times X \times X$ with a $G$-action given by
	\begin{align}
	g(x, y, z) & = (g(x), g(y), g(z)), \quad (x,y,z \in X,\ g \in G)
	\end{align}
	Let $\mathcal{O}(\Gamma)$ denote the set of all orbits from such $G$-action.
\end{definition}

By Definition \ref{defn: G action on X times X times X} and since \eqref{eqn: orthonormal basis of tensor V} is a basis of $V^{\otimes 3}$, each $g \in G$ may be viewed a linear transformation $g:V^{\otimes 3} \rightarrow V^{\otimes 3}$ such that
\begin{align*}
g(x \otimes y \otimes z) & = g(x) \otimes g(y) \otimes g(z), \quad (x,y,z \in X).
\end{align*}
In other words, $V^{\otimes 3}$ is a $G$-module. 
Define $\operatorname{Fix}(G) = \{v \in V^{\otimes 3}\ |\ g(v) = v\ \forall g \in G \}$. 

\begin{definition}\label{defn: characteristic vector}\rm \cite[Definition 10.5]{terwilligerS3}
	For each $\Omega \in \mathcal{O}(\Gamma)$, we define the vector
	$$\chi_{\Omega}  = \textstyle \sum_{(x,y,z) \in \Omega} x \otimes y \otimes z \in V^{\otimes 3}.$$
	We call $\chi_{\Omega}$ the characteristic vector of $\Omega$.
\end{definition}

By \cite[pp. 21--22]{terwilligerS3}, we have
\begin{enumerate}
	\item[(K1)] For $g \in G$ and $B \in \mathbb{T}$, we have $gB = Bg$ on $V^{\otimes 3}$,
	\item[(K2)] $\operatorname{Fix}(G)$ is a $\mathbb{T}$-submodule of $V^{\otimes 3}$,
	\item[(K3)] $\Lambda \subseteq \operatorname{Fix}(G)$,
	\item[(K4)] $\{\chi_{\Omega}\ |\ \Omega \in \mathcal{O}(\Gamma) \}$ is an orthogonal basis of $\operatorname{Fix}(G)$,
	\item[(K5)] $\operatorname{dim}(\operatorname{Fix}(G)) = |\mathcal{O}(\Gamma)|$.
\end{enumerate}

\section{The cycle graph $\Gamma_a$} \label{sect: cycles and their properties}

In this section, we investigate a family of $Q$-polynomial distance-regular graphs known as cycles. We begin with some notations.\\

For a fixed integer $D \geq 2$, define
\begin{align*}
	X_e = \{0,1,\ldots,2D-1\} \text{ and } X_o = \{0,1,\ldots,2D\}. 
\end{align*}
For $a \in \{o,e\}$,  let $\Gamma_a = (X_a,R_a)$ denote the graph with vertex set $X_a$ and edge set $R_a =
\{(x,y)\ |\ x,y \in X_a \text{ and } |x-y| \equiv \pm 1\ (\operatorname{mod}\ |X_a|) \}$. Observe that $\Gamma_a$ is a cycle on $|X_a|$ vertices with diameter $D$. For convenience, view $-1$ as the vertex $|X_a|-1$ and $|X_a|$ as the vertex $0$. For $z \in X_a$, the set of all neighbors of $z$ is given by
\begin{align}\label{eqn: neighbors of z}
\Gamma_a(z) & = \{z-1, z+1 \}.
\end{align}

\subsection{The eigenvalue sequences of $\Gamma_a$}\label{subsect: eigenvalue sequences of cycles}

For $a\in \{e,o\}$, let $A_0, A_1, \ldots, A_D$ and $E_0, E_1, \ldots, E_D$ denote the distance matrices and primitive idempotents of $\Gamma_a$, respectively. let $A_0^*, A_1^*, \ldots, A_D^*$ denote the dual distance matrices of $\Gamma_a$ with respect to vertex $0$. Let $\zeta_a$ denote a fixed primitive $|X_a|\textsuperscript{th}$ root of unity. Let $V_a$ denote the standard module of $\Gamma_a$ with orthonormal basis $X_a$.

\begin{lemma}\label{lem: eigenvectors of A1}
	For each $j\in X_a$, we define the vector $v_j := \textstyle |X_a|^{-1/2} \sum_{x \in X_a} \zeta_a^{xj}x \in V_a$. For convenience, we write $v_{|X_a|}:=v_0$. Then the following hold.
	\begin{enumerate}[(i)]
		\item $A_1v_j = (\zeta_a^{j}+\zeta_a^{-j})v_j$.
		\item If $a=e$, then $\{v_j,v_{2D-j}\}$ spans the eigenspace of $A_1$ corresponding to $(\zeta_e^{j}+\zeta_e^{-j})$.
		\item If $a=o$, then $\{v_j,v_{2D+1-j}\}$ spans the eigenspace of $A_1$ corresponding to $(\zeta_o^{j}+\zeta_o^{-j})$.
	\end{enumerate}
\end{lemma}

\begin{proof}
	Immediate from \eqref{eqn: action of A1 on each basis vector} and \eqref{eqn: neighbors of z}. 
\end{proof}

\begin{corollary}\label{cor: primitive idempotents of Gamma a}
	If $a = e$, the primitive idempotents $E_0, E_1, \ldots, E_D$ of $\Gamma_e$ are given by
	\begin{align*}
	E_0 & = \textstyle \frac{1}{2D} \sum_{x,y \in X_e} xy,\\
	E_j & = \textstyle \frac{1}{2D} \sum_{x,y \in X_e} (\zeta_e^{(x-y)j}+\zeta_e^{-(x-y)j})xy, \quad (j\in\{1,2,\ldots,D-1\}),\\
	E_n & = \textstyle \frac{1}{2D}\sum_{x,y \in X_e} \zeta_e^{(x-y)n}xy.
	\end{align*}
	If $a = o$, then the primitive idempotents $E_0, E_1, \ldots, E_D$ of $\Gamma_o$ are given by
	\begin{align*}
	E_0 & = \textstyle \frac{1}{2D+1} \sum_{x,y \in X_o} xy,\\
	E_j & = \textstyle \frac{1}{2D+1} \sum_{x,y \in X_o} (\zeta_o^{(x-y)j}+\zeta_o^{-(x-y)j})xy,\quad (j\in\{1,2,\ldots,D\}). 
	\end{align*}
\end{corollary}

\begin{proof}
	If $a = e$, we obtain the primitive idempotents using 
	\begin{align*}
	E_j & = \left\{\begin{array}{ll}
	v_j \bar{v}_{j}^t  & \text{ if } j \in \{0,D\},\\
	v_j\bar{v}_j^t+v_{2D-j}\bar{v}_{2D-j}^t & \text{ otherwise}.
	\end{array} \right.
	\end{align*}
	If $a = o$, we obtain the primitive idempotents using
	\begin{align*}
	E_j & = \left\{\begin{array}{ll}
	v_j \bar{v}_{j}^t  & \text{ if } j = 0,\\
	v_j\bar{v}_j^t+v_{2D+1-j}\bar{v}_{2D+1-j}^t & \text{ otherwise}.
	\end{array} \right.
	\end{align*} 
\end{proof}

\begin{corollary}\label{cor: dual adjacency matrices of Gamma a}
	We have $A_1^* = \textstyle \sum_{y \in X_a} (\zeta_a^{y}+\zeta_a^{-y})yy$.
\end{corollary}

\begin{proof}
	Immediate from \eqref{eqn: dual distance matrices of Gamma} and Corollary \ref{cor: primitive idempotents of Gamma a}.
\end{proof}

By Lemma \ref{lem: eigenvectors of A1} and Corollaries \ref{cor: primitive idempotents of Gamma a}--\ref{cor: dual adjacency matrices of Gamma a}, the following hold.
\begin{enumerate}
	\item[(L1)] $\Gamma_a$ has eigenvalue sequence $\{\theta_i\}_{i=0}^D$ with $\theta_i = \zeta_a^{i} + \zeta_a^{-i}$ for $i \in \{0,1,\ldots,D\}$. 
	\item[(L2)] $\Gamma_a$ has dual eigenvalue sequence $\{\theta_i^*\}_{i=0}^D$ with $\theta_i^* = \zeta_a^{i} + \zeta_a^{-i}$ for $i \in \{0,1,\ldots,D\}$. 
\end{enumerate}

By (L1)--(L2) and Lemma \ref{lem: the scalars beta, gamma, etc. from Q-poly DRG}, we have
\begin{enumerate}
	\item[(L3)] $\beta = \zeta_a + \zeta_a^{-1},\ \gamma = \gamma^* = 0,\  \varrho = \varrho^* = -(\zeta_a-\zeta_a^{-1})^2$. 
\end{enumerate}

\subsection{The automorphism group $\operatorname{Aut}(\Gamma_a)$}

The automorphism group $\operatorname{Aut}(\Gamma_a)$ of $\Gamma_a$ is the dihedral group consisting of $|X_a|$ rotation and $|X_a|$ reflection symmetries. Throughout, let $G_a = \operatorname{Aut}(\Gamma_a)$ and let $G_a'$ denote the abelian subgroup of $G_a$ consisting of rotation symmetries. Let $\mathcal{O}(\Gamma_a)$ denote the collection of orbits from the action of $G_a$ on $X_a \times X_a \times X_a$ (see Definition \ref{defn: G action on X times X times X}).

\begin{lemma}\label{lem: transitivity of Aut Gamma}
	For each $x \in X_a$, there exists an element $g \in G_a$ such that $g(x) = 0$. 	
\end{lemma}

\begin{proof}
	Immediate from the rotation symmetries.
\end{proof}

\begin{definition}\label{defn:profile}\rm 
	Let $x,y,z \in X_a$. The ordered triple $(x,y,z)$ is said to have a \emph{distance profile} $(i,j,k)$ whenever $\partial(y,z) = i$, $\partial(x,z) = j$, and $\partial(x,y) = k$. Let $\operatorname{DP}(\Gamma_a)$ denote the number of distinct distance profiles in $\Gamma_a$. 
\end{definition}

\begin{lemma}\label{lem: number of orbits of Aut Gamma}
	With above notation, $|\mathcal{O}(\Gamma_e)| = 2D^2+2$ and $|\mathcal{O}(\Gamma_o)| = 2D^2+2D+1$.
\end{lemma}

\begin{proof}
	Use Burnside's lemma.
\end{proof}

\begin{lemma}\label{lem: number of distance profiles}
	We have $\operatorname{DP}(\Gamma_e) = 2D^2+2$ and $\operatorname{DP}(\Gamma_o) = 2D^2+2D+1$.
\end{lemma}

\begin{proof}
	We prove $\operatorname{DP}(\Gamma_e) = 2D^2 + 2$. By Lemma \ref{lem: transitivity of Aut Gamma}, it is sufficient to count distance profiles of triples $(x,y,0)$ for $x,y \in X_e$. Observe that
	\begin{align}\label{eqn: index set for orbits in even cycles}
	O_e := \{(x,y,0)\ |\ 0 \leq x \leq D, y \in \{0,D\} \} \cup \{(x,y,0)\ |\ x \in X_e, 1 \leq y \leq D-1 \}
	\end{align}
	is a collection of triples with mutually distinct distance profiles. 
	We claim that this list completes all possible distance profiles in $\Gamma_e$. If $(x,y,0) \notin O_e$, then $(x',y',0):= (|X_e|-x,|X_e|-y,0) \in O_e$. Claim holds since the triples $(x,y,0)$ and $(x',y',0)$ have the same distance profile. To prove $\operatorname{DP}(\Gamma_o) = 2D^2+2D + 1$, we show that
	\begin{align}\label{eqn: index set for orbits in odd cycles}
	O_o := \{(x,0,0)\ |\ 0 \leq x \leq D \} \cup \{(x,y,0)\ |\ x \in X_e, 1 \leq y \leq D \}
	\end{align}
	is a collection of triples with mutually distinct distance profiles in a similar fashion.
\end{proof}

\begin{corollary}\label{cor: distance profiles of orbits}
	No two orbits have the same distance profile.
\end{corollary}

\begin{proof}
	Immediate from Lemmas \ref{lem: number of orbits of Aut Gamma}--\ref{lem: number of distance profiles}.
\end{proof}

\begin{corollary}\label{cor: the map phi}
	For $a \in \{e,o\}$, let $O_a$ be as in \eqref{eqn: index set for orbits in even cycles} or \eqref{eqn: index set for orbits in odd cycles}. Then there exists a surjective map $\phi: X_a \times X_a \times X_a \rightarrow O_a$ such that $(x,y,z),(x,y,z)^\phi$ belong to the same orbit in $\mathcal{O}(\Gamma_a)$. Moreover, $(x_1,y_1,z_1)^\phi = (x_2,y_2,z_2)^\phi$ if and only if $(x_1,y_1,z_1)$ and $(x_2,y_2,z_2)$ have the same distance profile.
\end{corollary}

\begin{proof}
	Clear from Lemma \ref{lem: number of orbits of Aut Gamma}, Corollary \ref{cor: distance profiles of orbits}, and the proof of Lemma \ref{lem: number of distance profiles}.
\end{proof}

\section{A basis for the fundamental $\mathbb{T}$-module}\label{sect: fundamental T-module of cycles}

For the rest of the paper, we adopt the following assumption.

\begin{assumption}\label{assump: assumptions}\rm 
	Fix an integer $D \geq 2$. For $a \in \{o,e\}$, let $\Gamma_a = (X_a,R_a)$ denote the cycle with vertex set $X_a$ and diameter $D$. Note that $\Gamma_a$ has eigenvalue sequences $\{\theta_i\}_{i=0}^D$ and $\{\theta_i^* \}_{i=0}^D$ given in (L1)--(L2). The scalars $\beta, \gamma, \gamma^*, \varrho, \varrho^*$ associated with $\Gamma_a$ are found in (L3). Let $\mathbb{T}=\mathbb{T}(\beta, \gamma, \gamma^*, \varrho, \varrho^*)$ be the corresponding $S_3$-symmetric tridiagonal algebra. Let $V_a$ denote the standard module of $\Gamma_a$ and let $\Lambda_a$ denote the fundamental $\mathbb{T}$-module in $V_a^{\otimes 3}$. Let $G_a = \operatorname{Aut}(\Gamma_a)$ and let $\mathcal{O}(\Gamma_a)$ denote the collection of all orbits from the action of $G_a$ on $X_a \times X_a \times X_a$. Let $O_a$ be as in \eqref{eqn: index set for orbits in even cycles} or \eqref{eqn: index set for orbits in odd cycles}. For $(x,y,0) \in O_a$, let $\Omega_a(x,y,0)$ denote the orbit in $\mathcal{O}(\Gamma_a)$ containing $(x,y,0)$. Finally, let $G'_a$ denote the subgroup of $G_a$ consisting of all rotation symmetries.
	
\end{assumption}

Here, we find a basis for $\Lambda_a$ made up of common eigenvectors for $A^{*(1)}, A^{*(2)}, A^{*(3)}$.

\begin{lemma}\label{lem: number of nonzero P h i j}
	With Assumption \ref{assump: assumptions}, we have
	\begin{align*}
	|\{P_{h,i,j} \neq \mathbf{0}\ |\ h,i,j \in \{0,1,\ldots,D\} \}| & = \left\{\begin{array}{ll}
	2D^2+2 & \text{if } a=e,\\
	2D^2+2D+1 & \text{if }a=o.
	\end{array} \right.
	\end{align*}
\end{lemma}

\begin{proof}
	Immediate from Definition \ref{defn: the vectors P h i j}, Definition \ref{defn:profile}, and Lemma \ref{lem: number of distance profiles}.
\end{proof}

\begin{lemma}\label{lem: each char vector is in the fundamental module}
	With Assumption \ref{assump: assumptions}, we have $\chi_{\Omega} \in \Lambda_a$ for each $\Omega \in \mathcal{O}(\Gamma_a)$.
\end{lemma}

\begin{proof}
	By Definition \ref{defn: characteristic vector} and Corollary \ref{cor: distance profiles of orbits}, each $\chi_{\Omega}$ is equal to a vector $P_{h,i,j} \neq \mathbf{0}$ for some $h,i,j \in \{0,1,\ldots,D\}$. By this and (I2), $\chi_{\Omega} \in \Lambda_a$ for $\Omega \in \mathcal{O}(\Gamma_a)$.
\end{proof}

\begin{theorem}\label{thm: basis of the fundamental module}
	With Assumption \ref{assump: assumptions}, we have $\Lambda_a = \operatorname{Fix}(G_a)$. Furthermore,
\begin{align*}
\operatorname{dim}(\Lambda_a) & = \left\{\begin{array}{ll}
\textstyle  2D^2+2 & \text{if } a = e,\\
\textstyle  2D^2+2D+1 & \text{if } a=o.
\end{array} \right.
\end{align*}
\end{theorem}

\begin{proof}
	By (K3)--(K4) and Lemma \ref{lem: each char vector is in the fundamental module}, we have $\Lambda_a = \operatorname{Fix}(G_a)$. Consequently, we have $\operatorname{dim}(\Lambda_a) = |\mathcal{O}(\Gamma_a)|$. The remaining assertion follows from Lemma \ref{lem: number of orbits of Aut Gamma}.
\end{proof}

\begin{corollary}\label{cor: nonzero p h i j}
	With Assumption \ref{assump: assumptions}, $\{P_{h,i,j} \neq \mathbf{0}\ |\ h,i,j \in \{0,1,\ldots,D\}  \}$ is a basis of $\Lambda_a$. Moreover, for $h,i,j \in \{0,1,\ldots,D\}$, we have
	\begin{align*}
	A^{*(1)}P_{h,i,j} & = \theta_h^*P_{h,i,j},\\
	A^{*(2)}P_{h,i,j} & = \theta_i^*P_{h,i,j},\\
	A^{*(3)}P_{h,i,j} & = \theta_j^*P_{h,i,j}.
	\end{align*}
\end{corollary}

\begin{proof}
	By (I2)--(I3), the set $\{P_{h,i,j} \neq \mathbf{0}\ |\ h,i,j \in \{0,1,\ldots,D\}\}$ is linearly independent in $\Lambda_a$. By Lemma \ref{lem: number of nonzero P h i j} and Theorem \ref{thm: basis of the fundamental module}, this set is a basis of $\Lambda_a$. The equations follow immediately from Definitions \ref{defn: the dual adjacency maps on tensor V} and \ref{defn: the vectors P h i j}.
\end{proof}

\begin{remark}\label{remk: Problem 12.5}\rm
		Recall the Terwilliger algebra of a $Q$-polynomial distance-regular graph with respect to a fixed vertex (see Subsection \ref{subsection: Terwilliger algebra of Gamma}). Let $\mathsf{T} = \mathsf{T}(x)$ denote the Terwilliger algebra of $\Gamma_a$ with respect to $x \in X_a$. By \cite[p. 206]{terwilligerSUBCONS}, we have
		\begin{align*}
		\operatorname{dim}(\mathsf{T}) & = \left\{\begin{array}{ll}
		\textstyle  2D^2+2 & \text{if } a = e,\\
		\textstyle  2D^2+2D+1 & \text{if } a=o.
		\end{array} \right.
		\end{align*}
		By Theorem \ref{thm: basis of the fundamental module}, there exists a vector space isomorphism from $\mathsf{T}$ to the fundamental $\mathbb{T}$-module $\Lambda_a$. This vector space isomorphism does not exist for general $Q$-polynomial distance-regular graphs (see \cite[Problem 12.5]{terwilligerS3} and the comments before it).
\end{remark}

\section{The action of the generators of $\mathbb{T}$ on $\chi_{\Omega}$}

Recall that $\mathbb{T}$ is generated by $A_1, A_2, A_3, A_1^*, A_2^*, A_3^*$. With Assumption \ref{assump: assumptions}, the generators $A_1^*, A_2^*, A_3^*$ act on $V_a^{\otimes 3}$ as $A^{*(1)}, A^{*(2)}, A^{*(3)}$, respectively while generators $A_1, A_2, A_3$ act on $V_a^{\otimes 3}$ as $A^{(1)}, A^{(2)}, A^{(3)}$, respectively (see Theorem \ref{thm: S3-symmetric tridiagonal algebra acting on tensor power}). In this section, we investigate the action of $A^{*(1)}, A^{*(2)}, A^{*(3)}, A^{(1)}, A^{(2)}, A^{(3)}$ on $\chi_{\Omega}$ for $\Omega \in \mathcal{O}(\Gamma_a)$. 

\begin{lemma}\label{lem: action of A*'s on Omega a x 0 0}
	With Assumption \ref{assump: assumptions}, let $(x,0,0) \in O_a$. Then
	\begin{align}
	A^{*(1)} \chi_{\Omega_a(x,0,0)} & = \theta_0^* \chi_{\Omega_a(x,0,0)}, \label{eqn: action of A* on x 0 0 first}\\
	A^{*(2)} \chi_{\Omega_a(x,0,0)} & = \theta_{x}^*\chi_{\Omega_a(x,0,0)},\\
	A^{*(3)} \chi_{\Omega_a(x,0,0)} & = \theta_{x}^*\chi_{\Omega_a(x,0,0)}. \label{eqn: action of A* on x 0 0 last}
	\end{align}
\end{lemma}

\begin{proof}
	Since $(x,0,0) \in \Omega_a(x,0,0)$, each triple in $\Omega_a(x,0,0)$ has distance profile $(0,x,x)$. In particular, $\chi_{\Omega_a(x,0,0)} = P_{0,x,x}$. By Corollary \ref{cor: nonzero p h i j}, \eqref{eqn: action of A* on x 0 0 first}--\eqref{eqn: action of A* on x 0 0 last} hold.
\end{proof}

\begin{lemma}\label{lem: action of A*'s on Omega a x y 0}
	With Assumption \ref{assump: assumptions}, let $(x,y,0) \in O_a$ with $1 \leq y \leq D$. Then
	\begin{align}
	A^{*(1)} \chi_{\Omega_a(x,y,0)} & = \theta_y^* \chi_{\Omega_a(x,y,0)}, \label{eqn: 11}\\
	A^{*(2)} \chi_{\Omega_a(x,y,0)} & = \left\{\begin{array}{ll}
	\theta_x^*\chi_{\Omega_a(x,y,0)} & \text{if } 0 \leq x \leq D,\\
	\theta_{|X_a|-x}^*\chi_{\Omega_a(x,y,0)} & \text{if } D+1 \leq x \leq |X_a|-1,
	\end{array} \right.\\
	A^{*(3)} \chi_{\Omega_a(x,y,0)} & = \left\{ \begin{array}{ll}
	\theta_{|x-y|}^* \chi_{\Omega_a(x,y,0)} & \text{ if } 0 \leq x \leq D,\\
	\theta_{\operatorname{min}\{x-y,|X_a|-x+y\}}^* \chi_{\Omega_a(x,y,0)} & \text{ if } D+1 \leq x \leq |X_a|-1.
	\end{array}\right. \label{eqn: 13}
	\end{align}
\end{lemma}

\begin{proof}
	Note that $(x,y,0) \in \Omega_a(x,y,0)$. If $0 \leq x \leq D$, then $\Omega_a(x,y,0)$ has distance profile $(y,x,|x-y|)$ and $\chi_{\Omega_a(x,y,0)} = P_{y,x,|x-y|}$. If $D+1 \leq x \leq |X_a|-1$, then $\Omega_a(x,y,0)$ has distance profile $(r,s,t)$ and $\chi_{\Omega_a(x,y,0)} = P_{r,s,t}$ where $$(r,s,t) = (y,|X_a|-x,\operatorname{min}\{x-y,|X_a|-x+y\}).$$ By Corollary \ref{cor: nonzero p h i j}, \eqref{eqn: 11}--\eqref{eqn: 13} hold.
\end{proof}

\begin{lemma}\label{lem: number of elements in each orbit}
	With Assumption \ref{assump: assumptions}, let $(x,y,0) \in O_a$. Then the following hold.
	\begin{enumerate}[(i)]
		\item If $a = e$, then 
		\begin{align}
		|\Omega_e(x,y,0)| & = \left\{\begin{array}{ll}
		|X_e| & \text{if } x,y \in \{0,D\},\\
		2|X_e| & \text{otherwise}. 
		\end{array} \right.
		\end{align}
		\item If $a = o$, then 
		\begin{align}
		|\Omega_o(x,y,0)| & = \left\{\begin{array}{ll}
		|X_o| & \text{if } x=0 \text{ and } y=0,\\
		2|X_o| & \text{otherwise}.
		\end{array} \right.
		\end{align}
	\end{enumerate}
\end{lemma}

\begin{proof}
	Let $(x,y,0) \in O_e$. Observe that $\Omega_e(x,y,0)$ is equal to $$\{g(x,y,0)\ |\ g\in G'_e \} \cup \{g(2D-x,2D-y,0)\ |\ g \in G'_e \}.$$ If $x,y \in \{0,D\}$, then the triple $(x,y,0)$ is the same as $(2D-x,2D-y,0)$ since $2D$ is treated as vertex $0$. Hence, $|\Omega_e(x,y,0)| = |G'_e| = |X_e|$. Otherwise, there is no $g \in G'_e$ satisfying $g(x,y,0) = (2D-x,2D-y,0)$. Consequently, $|\Omega_e(x,y,0)| = 2|G'_e| = 2|X_e|$. This proves (i). We prove (ii) using a similar argument.
\end{proof}

\begin{lemma}\label{lem: properties of the map phi}
	With Assumption \ref{assump: assumptions}, let $(x,y,0) \in O_a$. Assume $\phi$ is the map in Corollary \ref{cor: the map phi}. Then the following hold.
	\begin{enumerate}[(i)]
		\item If $|\Omega_a(x,y,0)| = |X_a|$, then
		\begin{align}
		(x-1,y,0)^\phi & = (x+1,y,0)^\phi, \label{eqn: equality of x-1 y 0}\\
		(x,y-1,0)^\phi & = (x,y+1,0)^\phi, \label{eqn: equality of x y-1 0}\\
		(x,y,|X_a|-1)^\phi & = (x,y,1)^\phi. \label{eqn: equality of x y 1}
		\end{align}
		\item If $|\Omega_a(x,y,0)| = 2|X_a|$, then each of the equations \eqref{eqn: equality of x-1 y 0}--\eqref{eqn: equality of x y 1} does not hold.
	\end{enumerate}
\end{lemma}

\begin{proof}
	We prove (i). Assume $|\Omega_a(x,y,0)| = |X_a|$. If $a = e$, then $x,y \in \{0,D\}$. If $a = o$, then $x = y = 0$. By this, \eqref{eqn: equality of x-1 y 0}--\eqref{eqn: equality of x y 1} hold by comparing distance profiles. We prove (ii). Assume $|\Omega_a(x,y,0)| = 2|X_a|$. If $a = e$, then either $1 \leq x \leq D-1$ or $1 \leq y \leq D-1$. If $a = o$, then either $1 \leq x \leq D$ or $1 \leq y \leq D$. By this, each of the equations \eqref{eqn: equality of x-1 y 0}--\eqref{eqn: equality of x y 1} does not hold since they have different distance profiles.
\end{proof}

\begin{lemma}\label{lem: action of the adjacency map on corner points} With Assumption \ref{assump: assumptions}, let $(x,y,0) \in O_a$ such that $|\Omega_a(x,y,0)| = |X_a|$. Then
		\begin{align*}
		A^{(1)}\chi_{\Omega_a(x,y,0)} & = \chi_{\Omega_a(x-1,y,0)^\phi} = \chi_{\Omega_a(x+1,y,0)^\phi},\\
		A^{(2)}\chi_{\Omega_a(x,y,0)} & = \chi_{\Omega_a(x,y-1,0)^\phi} = \chi_{\Omega_a(x,y+1,0)^\phi},\\
		A^{(3)}\chi_{\Omega_a(x,y,0)} & = \chi_{\Omega_a(x,y,|X_a|-1)^\phi} = \chi_{\Omega_a(x,y,1)^\phi}.
		\end{align*}
\end{lemma}

\begin{proof}
	By assumption on $(x,y,0)$, we have $\chi_{\Omega_a(x,y,0)} = \textstyle \sum_{g \in G'_a} g(x \otimes y \otimes 0)$. By (K1), Definition \ref{defn: the adjacency maps on tensor V}, \eqref{eqn: action of A1 on each basis vector}, and \eqref{eqn: neighbors of z}, we have
	\begin{align*}
	A^{(1)}\chi_{\Omega_a(x,y,0)} & = \textstyle \sum_{g \in G'_a}g((x-1)\otimes y \otimes 0) + \sum_{g \in G'_a}g((x+1)\otimes y \otimes 0)\\
	& = \chi_{\Omega_a(x-1,y,0)^\phi}\\ 
	& = \chi_{\Omega_a(x+1,y,0)^\phi},
	\end{align*}
	where the last line is due to Lemma \ref{lem: properties of the map phi}(i). This proves the first equation. The remaining equations are proven similarly.
\end{proof}

\begin{lemma}\label{lem: action of the adjacency map on non-corner points}
	With Assumption \ref{assump: assumptions}, let $(x,y,0) \in O_a$ such that $|\Omega_a(x,y,0)| = 2|X_a|$. Then
		\begin{align*}
		A^{(1)}\chi_{\Omega_a(x,y,0)} & = \textstyle \frac{2|X_a|}{|\Omega_a(x-1,y,0)^\phi|}\chi_{\Omega_a(x-1,y,0)^\phi} + \frac{2|X_a|}{|\Omega_a(x+1,y,0)^\phi|} \chi_{\Omega_a(x+1,y,0)^\phi},\\
		A^{(2)}\chi_{\Omega_a(x,y,0)} & = \textstyle \frac{2|X_a|}{|\Omega_a(x,y-1,0)^\phi|}\chi_{\Omega_a(x,y-1,0)^\phi} + \frac{2|X_a|}{|\Omega_a(x,y+1,0)^\phi|} \chi_{\Omega_a(x,y+1,0)^\phi},\\
		A^{(3)}\chi_{\Omega_a(x,y,0)} & = \textstyle \frac{2|X_a|}{|\Omega_a(x,y,|X_a|-1)^\phi|}\chi_{\Omega_a(x,y,|X_a|-1)^\phi} + \frac{2|X_a|}{|\Omega_a(x,y,1)^\phi|} \chi_{\Omega_a(x,y,1)^\phi}.
		\end{align*}
\end{lemma}

\begin{proof}
	By assumption on $(x,y,0)$, we see that $\chi_{\Omega_a(x,y,0)}$ is equal to $$\textstyle \sum_{g \in G'_a} g(x \otimes y \otimes 0)+\sum_{g \in G'_a} g((|X_a|-x) \otimes (|X_a|-y) \otimes 0).$$ 
	By (K1), Definition \ref{defn: the adjacency maps on tensor V}, \eqref{eqn: action of A1 on each basis vector}, and \eqref{eqn: neighbors of z}, we have
	\begin{align}
	& A^{(1)} \chi_{\Omega_a(x,y,0)} \notag \\  
	& = \textstyle \sum_{g \in G'_a}  g((x-1)\otimes y \otimes 0) + \sum_{g \in G'_a}  g((|X_a|-x+1)\otimes (|X_a|-y) \otimes 0) \label{eqn: x-1, y, 0}\\
	& \quad + \textstyle \sum_{g \in G'_a}g((x+1)\otimes y \otimes 0) + \textstyle \sum_{g \in G'_a}g((|X_a|-x-1)\otimes (|X_a|-y) \otimes 0). \label{eqn: x+1, y, 0}
	\end{align}
	If $|\Omega_a(x-1,y,0)^\phi| = 2|X_a|$ (resp. $|\Omega_a(x-1,y,0)^\phi| = |X_a|$), we find that \eqref{eqn: x-1, y, 0} is equal to $\chi_{\Omega_a(x-1,y,0)^\phi}$ (resp. $2\chi_{\Omega_a(x-1,y,0)^\phi}$). Similarly, \eqref{eqn: x+1, y, 0} is equal to either $\chi_{\Omega_a(x+1,y,0)^\phi}$ or $2\chi_{\Omega_a(x+1,y,0)^\phi}$. This gives the first equation. We prove the other equations similarly.
\end{proof}

We illustrate the action of $A^{(1)}, A^{(2)}, A^{(3)}$ on $\chi_{\Omega}$ for cycles with few vertices.

\begin{example}\rm 
	Let $D = 2$ and consider the cycle $\Gamma_e=(X_e,R_e)$ on four vertices. Then we have $\operatorname{dim}(\Lambda_e) = 10$. Furthermore, the basis $\{\chi_{\Omega}\ |\ \Omega \in \mathcal{O}(\Gamma_e) \}$ can be identified with the following 10 points on the plane where each point corresponds to a unique $\chi_{\Omega}$. 
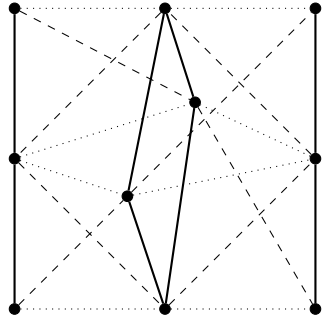
		\begin{center}
		\begin{tikzpicture}
		\tikzset{mynode/.style={circle, draw, fill=black, minimum size=4pt, inner sep=0pt}}
		
		\node[mynode] (v1) at (0,0) {};
		\node[mynode] (v2) at (0,2) {};
		\node[mynode] (v3) at (0,4) {};
		
		\node[mynode] (v4) at (1.5,1.5) {};
		\node[mynode] (v5) at (4,2) {};
		\node[mynode] (v6) at (2.4,2.75) {};
		
		\node[mynode] (v8) at (2,0) {};
		\node[mynode] (v9) at (4,0) {};
		\node[mynode] (v10) at (2,4) {};
		\node[mynode] (v11) at (4,4) {};
		
		\draw[thick]     (v1) -- (v2);
		\draw[thick]     (v2) -- (v3);
		\draw[thick]     (v8) -- (v4);
		\draw[thick]     (v4) -- (v10);
		\draw[thick]     (v10) -- (v6);
		\draw[thick]     (v6) -- (v8);
		\draw[thick]     (v11) -- (v5);
		\draw[thick]     (v5) -- (v9);
		\draw[dotted]     (v1) -- (v8);
		\draw[dotted]     (v8) -- (v9);
		\draw[dotted]     (v2) -- (v4);
		\draw[dotted]     (v4) -- (v5);
		\draw[dotted]     (v5) -- (v6);
		\draw[dotted]     (v2) -- (v6);
		\draw[dotted]     (v3) -- (v10);
		\draw[dotted]     (v10) -- (v11);
		\draw[dashed]     (v1) -- (v4);
		\draw[dashed]     (v4) -- (v11);
		\draw[dashed]     (v2) -- (v8);
		\draw[dashed]     (v8) -- (v5);
		\draw[dashed]     (v5) -- (v10);
		\draw[dashed]     (v10) -- (v2);
		\draw[dashed]     (v3) -- (v6);
		\draw[dashed]     (v6) -- (v9);
		\end{tikzpicture}
		
	\end{center}
	We describe the action of $A^{(1)}, A^{(2)}, A^{(3)}$ on a corner point. For $i \in \{1,2,3\}$, the map $A^{(i)}$ sends a corner point to one adjacent point. Specifically, the map $A^{(1)}$ sends it to an adjacent point whose edge is dotted. On the other hand, the map $A^{(2)}$ sends it to an adjacent point whose edge is solid while $A^{(3)}$ to an adjacent point whose edge is dashed. We describe the action of $A^{(1)}, A^{(2)}, A^{(3)}$ on a non-corner point. For $i \in \{1,2,3\}$, the map $A^{(i)}$ sends a non-corner point to two adjacent points. Specifically, the map $A^{(1)}$ sends it to adjacent points whose edges are dotted. On the other hand, the map $A^{(2)}$ sends it to adjacent points whose edge are solid while $A^{(3)}$ to adjacent points whose edge are dashed. If the adjacent point happens to be a corner, the coefficient of the corner point is 2. Otherwise, the coefficient is 1. We can draw an analogous diagram for the case $D \geq 3$ to describe the action of $A^{(1)}, A^{(2)}, A^{(3)}$.

\end{example}

\begin{example}\rm 
	Let $D = 2$ and consider the cycle $\Gamma_o=(X_o,R_o)$ on five vertices. Then $O_o = \{(0,0,0), (1,0,0), (2,0,0) \} \cup \{(x,1,0)\ |\ x \in X_o \} \cup \{(x,2,0)\ |\ x \in X_o \}$. We describe the action of $A^{(1)}, A^{(2)}, A^{(3)}$ on $\chi_{\Omega_o(0,0,0)}$. Since $|\Omega_o(0,0,0)| = 5 = |X_o|$, we have
	\begin{align*}
	A^{(1)}\chi_{\Omega_o(0,0,0)} & = \chi_{\Omega_o(1,0,0)},\\
	A^{(2)}\chi_{\Omega_o(0,0,0)} & = \chi_{\Omega_o(0,1,0)},\\
	A^{(3)}\chi_{\Omega_o(0,0,0)} & = \chi_{\Omega_o(1,1,0)},
	\end{align*}
	by Lemma \ref{lem: action of the adjacency map on corner points}. We show the action of $A^{(1)}, A^{(2)}, A^{(3)}$ on $\chi_{\Omega_o(0,1,0)}$. Since $|\Omega_o(0,1,0)| = 2|X_o|$, we obtain
	\begin{align*}
	A^{(1)}\chi_{\Omega_o(0,1,0)} & = \chi_{\Omega_o(4,1,0)}+\chi_{\Omega_o(1,1,0)},\\
	A^{(2)}\chi_{\Omega_o(0,1,0)} & = 2\chi_{\Omega_o(0,0,0)}+\chi_{\Omega_o(0,2,0)},\\
	A^{(3)}\chi_{\Omega_o(0,1,0)} & = \chi_{\Omega_o(1,0,0)}+\chi_{\Omega_o(1,2,0)},
	\end{align*}
	by Lemma \ref{lem: action of the adjacency map on non-corner points}.
\end{example}

\section{Another basis of the fundamental module}\label{sect: another basis of the fundamental module}

In Section \ref{sect: fundamental T-module of cycles}, we found a basis for $\Lambda_a$ consisting of common eigenvectors of the linear maps $A^{*(1)}, A^{*(2)}, A^{*(3)}$. In this section, we find a basis for $\Lambda_a$ consisting of common eigenvectors of linear maps $A^{(1)}, A^{(2)}, A^{(3)}$.\\

	Recall the scalars $q^{h}_{ij}$ from \eqref{eqn: the scalars q h i j}. 
	
\begin{lemma}\label{lem: number of nonzero q h i j}
	With Assumption \ref{assump: assumptions}, we have
	\begin{align}\label{eqn: number of nonzero q h i j}
	|\{q_{ij}^h \neq 0\ |\ h,i,j \in \{0,1,\ldots,D\} \}| & = \operatorname{dim}(\Lambda_a).
	\end{align}
\end{lemma}

\begin{proof}
	We count the left-hand side of \eqref{eqn: number of nonzero q h i j} using Corollary \ref{cor: primitive idempotents of Gamma a}, and \eqref{eqn: the scalars q h i j}. We prove equality using Theorem \ref{thm: basis of the fundamental module}.
\end{proof}

\begin{corollary}\label{cor: another basis of Lambda a}
	With Assumption \ref{assump: assumptions}, $\{Q_{h,i,j} \neq \mathbf{0} \ |\ h,i,j \in \{0,1,\ldots,D\} \}$ is a basis of $\Lambda_a$. Moreover, for $h,i,j \in \{0,1,\ldots,D\}$, we have
	\begin{align*}
	A^{(1)}Q_{h,i,j} & = \theta_h Q_{h,i,j},\\
	A^{(2)}Q_{h,i,j} & = \theta_i Q_{h,i,j},\\
	A^{(3)}Q_{h,i,j} & = \theta_j Q_{h,i,j}.
	\end{align*}
\end{corollary}

\begin{proof}
	By (J2)--(J3), $\{Q_{h,i,j} \neq \mathbf{0} \ |\ h,i,j \in \{0,1,\ldots,D\} \}$ is linearly independent in $\Lambda_a$. By (J4) and Lemma \ref{lem: number of nonzero q h i j}, this set is a basis of $\Lambda_a$. The above equations hold by Definition \ref{defn: the vectors Q h i j}, Definition \ref{defn: the adjacency maps on tensor V}, (C5), and since $A_1 = \sum_{k=0}^D \theta_k E_k$.
\end{proof}

\section{Writing the $Q_{h,i,j}$ in terms of the $P_{r,s,t}$}

Recall that $\{\chi_{\Omega}\ |\ \Omega \in \mathcal{O}(\Gamma_a) \}$ is a basis of $\Lambda_a$ and each $\chi_{\Omega}$ is equal to a vector $P_{h,i,j}$ for some $h,i,j \in \{0,1,\ldots,D\}$. In Corollary \ref{cor: another basis of Lambda a}, $\{Q_{h,i,j} \neq \mathbf{0}\ |\ h,i,j \in \{0,1,\ldots,D\} \}$ is also a basis of $\Lambda_a$. In this section, we aim to write the $Q_{h,i,j}$ in terms of the $P_{h,i,j}$.\\

For convenience, let $y$ denote the basis vector of $V_a$ corresponding to vertex $0$.

\begin{lemma}\label{lem: action of the primitive idempotents on the basis vector 0} 
	With Assumption \ref{assump: assumptions}, let $E_0, E_1, \ldots, E_D$ denote the primitive idempotents of $\Gamma_a$. If $a=e$, then
	\begin{align*}
	E_0y & = \textstyle \frac{1}{2D} \sum_{x \in X_e} x,\\
	E_jy & = \textstyle \frac{1}{2D} \sum_{x \in X_e} (\zeta_e^{xj}+\zeta_e^{-xj})x, \quad (j\in\{1,2,\ldots,D-1\}),\\
	E_Dy & = \textstyle \frac{1}{2D}\sum_{x \in X_e} \zeta_e^{xD}x.
	\end{align*}
	If $a = o$, then
	\begin{align*}
	E_0y & = \textstyle \frac{1}{2D+1} \sum_{x \in X_o} x,\\
	E_jy & = \textstyle \frac{1}{2D+1} \sum_{x \in X_o} (\zeta_o^{xj}+\zeta_o^{-xj})x,\quad (j\in\{1,2,\ldots,D\}). 
	\end{align*}
\end{lemma}

\begin{proof}
	Immediate from Corollary \ref{cor: primitive idempotents of Gamma a}.
\end{proof}

\begin{definition}\label{defn: the vectors Q h i j prime}\rm With Assumption \ref{assump: assumptions} we define for $k \in \{0,1,\ldots,D\}$ and $x \in X_a$
	\begin{align*}
	\zeta_a(x,k) & = \left\{\begin{array}{ll}
	1 & \text{if } k=0,\\
	\zeta_e^{xD} & \text{if } a=e \text{ and } k=D,\\
	\zeta_e^{xk}+\zeta_e^{-xk} & \text{if } a=e \text{ and } k \notin \{0,D\},\\
	\zeta_o^{xk}+\zeta_o^{-xk} & \text{if } a=o \text{ and } k \neq 0.
	\end{array} \right.
	\end{align*}
	By Lemma \ref{lem: action of the primitive idempotents on the basis vector 0} observe that $E_ky = \textstyle |X_a|^{-1} \sum_{x \in X_a} \zeta_a(x,k)x$ for $k \in \{0,1,\ldots,D\}$.
\end{definition}

\begin{lemma}\label{lem: symmetry of zeta a}
	With Assumption \ref{assump: assumptions}, we have
	\begin{align*}
	\zeta_a(x,k) & = \zeta_a(|X_a|-x,k), \quad (x \in X_a, k \in \{0,1,\ldots,D\}).
	\end{align*}
\end{lemma}

\begin{proof}
	Follows from the fact that $\zeta_a$ is a fixed primitive $|X_a|^{\textsuperscript{th}}$ root of unity.
\end{proof}

\begin{lemma}
With Assumption \ref{assump: assumptions}, let $h,i,j \in \{0,1,\ldots,D\}$ be given. Then
	\begin{align}\label{eqn: expansion of Q h i j}
	Q_{h,i,j} & = \textstyle |X_a|^{-2} \sum_{g \in G_a'}\sum_{u,w,z \in X_a} \zeta_a(u,h)\zeta_a(w,i)\zeta_a(z,j)\ g(u \otimes w \otimes z).
	\end{align}
\end{lemma}

\begin{proof}
	By (J1) and (K1), we have
	\begin{align*}
	Q_{h,i,j} & = |X_a|E_h^{(1)}E_i^{(2)}E_k^{(3)}P_{0,0,0}\\
	& = \textstyle |X_a|E_h^{(1)}E_i^{(2)}E_k^{(3)} \sum_{g \in G_a'}g(y \otimes y \otimes y)\\
	& = \textstyle |X_a|\sum_{g \in G_a'}g(E_hy \otimes E_iy \otimes E_jy).
	\end{align*}
	Statement holds by the comment below Definition \ref{defn: the vectors Q h i j prime}. 
\end{proof}

\begin{proposition}\label{prop: transition from char vector to Q h i j}
	With Assumption \ref{assump: assumptions}, let $h,i,j \in \{0,1,\ldots,D\}$ be given. Then
	\begin{align*}
	Q_{h,i,j} & = |X_a|^{-2}\textstyle \sum_{\Omega \in \mathcal{O}(\Gamma_a)}\frac{|X_a|}{|\Omega|} (\sum_{(u,w,z) \in \Omega} \zeta_a(u,h)\zeta_a(w,i)\zeta_a(z,j)) \chi_{\Omega}.
	\end{align*}
\end{proposition}

\begin{proof}
	Pick $\Omega \in \mathcal{O}(\Gamma_a)$. By Lemma \ref{lem: number of elements in each orbit}, either $|\Omega| = |X_a|$ or $|\Omega|=2|X_a|$. Suppose $|\Omega| = |X_a|$. Then for each $(u,w,z) \in \Omega$, we have
	\begin{align*}
	& \textstyle \sum_{g \in G_a'}  \zeta_a(u,h)\zeta_a(w,i)\zeta_a(z,j)\ g(u \otimes w \otimes z)\\ 
	& = \zeta_a(u,h)\zeta_a(w,i)\zeta_a(z,j) \textstyle \sum_{g \in G_a'} g(u \otimes w \otimes z)\\
	& = \zeta_a(u,h)\zeta_a(w,i)\zeta_a(z,j) \chi_{\Omega}.
	\end{align*}
	Thus, the coefficient of $\chi_{\Omega}$ in \eqref{eqn: expansion of Q h i j} is $|X_a|^{-2}\sum_{(u,w,z) \in \Omega} \zeta_a(u,h)\zeta_a(w,i)\zeta_a(z,j)$. Now, we assume $|\Omega| = 2|X_a|$. If $(u,w,z) \in \Omega$, then we find that $(|X_a|-u,|X_a|-w,|X_a|-z) \in \Omega$ where $|X_a|$ is treated as vertex $0$. Also, $(|X_a|-u,|X_a|-w,|X_a|-z) \neq g(u,w,z)$ $\forall g \in G_a'$. By this, we find a partition $\Omega = \Omega' \cup \Omega''$ such that each of $\Omega'$ and $\Omega''$ is $G_a'$-invariant. In particular, the map $(u,w,z) \mapsto (u',w',z') := (|X_a|-u,|X_a|-w,|X_a|-z)$ is a bijection from $\Omega'$ to $\Omega''$. By Lemma \ref{lem: symmetry of zeta a}, for $(u,w,z) \in \Omega'$, we have
	\begin{align*}
	& \textstyle \sum_{g \in G'_a}  \zeta_a(u,h)\zeta_a(w,i)\zeta_a(z,j)\ g(u \otimes w \otimes z)\\ 
	& \qquad + \textstyle \sum_{g \in G'_a} \zeta_a(u',h)\zeta_a(w',i)\zeta_a(z',j)\ g(u' \otimes w' \otimes z')\\
	& = \zeta_a(u,h)\zeta_a(w,i)\zeta_a(z,j) \textstyle \sum_{g \in G'_a} g(u \otimes w \otimes z)\\
	& \qquad + \textstyle \zeta_a(u',h)\zeta_a(w',i)\zeta_a(z',j)\sum_{g \in G'_a} g(u' \otimes w' \otimes z')\\
	& = \textstyle \frac{1}{2} \times (\zeta_a(u,h)\zeta_a(w,i)\zeta_a(z,j)+\zeta_a(u',h)\zeta_a(w',i)\zeta_a(z',j)) \chi_{\Omega}.
	\end{align*}
	Therefore, the coefficient of $\chi_{\Omega}$ in \eqref{eqn: expansion of Q h i j} is $|X_a|^{-2} \times \frac{1}{2}\sum_{(u,w,z) \in \Omega} \zeta_a(u,h)\zeta_a(w,i)\zeta_a(z,j)$. Statement holds.
\end{proof}

\begin{corollary}
	\label{cor: matrix of transition from P h i j to Q h i j}
	With Assumption \ref{assump: assumptions}, we define $\Omega(r,s,t)$ as the collection of all triples $(x,y,z) \in X_a \times X_a \times X_a$ whose distance profile is $(r,s,t)$. Then we have
	\begin{align}
	|X_a|\langle Q_{h,i,j},P_{r,s,t} \rangle = \textstyle |\Omega(r,s,t)|^{-1} \sum_{(u,w,z) \in \Omega(r,s,t)} \zeta_a(u,h)\zeta_a(w,i)\zeta_a(z,j)\| P_{r,s,t}\|^2
	\end{align}
	for $h,i,j,r,s,t \in \{0,1,\ldots,D\}$ where $\langle\ ,\ \rangle$ is the unique Hermitian form on $V_a^{\otimes 3}$ with respect to which the basis $\{x \otimes y \otimes z\ |\  x,y,z \in X_a \}$ is orthonormal.
\end{corollary}

\begin{proof}
	Follows from Proposition \ref{prop: transition from char vector to Q h i j}, (I3), and since each $\chi_{\Omega}$ is equal to some $P_{r,s,t} \neq \mathbf{0}$. 
\end{proof}

	Note that the value of $\|P_{r,s,t}\|^2$ can be obtained using \cite[Lemma 9.11]{terwilligerS3}.

\section{Some comments}

To summarize, we investigated a distance-regular graph $\Gamma_a = (X_a,R_a)$ with vertex set $X_a$ and edge set $R_a$. The graph $\Gamma_a$ is a cycle with diameter $D \geq 2$ and is $Q$-polynomial with respect to the ordering $E_0, E_1, \ldots, E_{D}$ of primitive idempotents. By Lemma \ref{lem: the scalars beta, gamma, etc. from Q-poly DRG}, we obtain scalars $\beta, \gamma, \gamma^*, \varrho, \varrho^*$ from the eigenvalue sequences $\{\theta_i\}_{i=0}^D, \{\theta_i^* \}_{i=0}^D$ of $\Gamma_a$. Then we considered $S_3$-symmetric tridiagonal algebra $\mathbb{T} = \mathbb{T}(\beta, \gamma, \gamma^*, \varrho, \varrho^*)$. Let $V_a$ denote the standard module of $\Gamma_a$. By Theorem \ref{thm: S3-symmetric tridiagonal algebra acting on tensor power}, $V_a^{\otimes 3}$ turns into a $\mathbb{T}$-module via the maps $A^{(1)}, A^{(2)}, A^{(3)}, A^{*(1)}, A^{*(2)}, A^{*(3)}$. Within $V^{\otimes 3}$, there is a unique irreducible $\mathbb{T}$-module referred to as the fundamental $\mathbb{T}$-module $\Lambda_a$. We have shown that each of the following is a basis for $\Lambda_a$.
\begin{enumerate}[(i)]
	\item $\{\chi_{\Omega}\ |\ \Omega \in \mathcal{O}(\Gamma_a) \}$,
	\item $\{P_{h,i,j} \neq \mathbf{0}\ |\ h,i,j \in \{0,1,\ldots,D\} \}$,
	\item $\{Q_{h,i,j} \neq \mathbf{0}\ |\ h,i,j \in \{0,1,\ldots,D\} \}$.
\end{enumerate}
Here, we explore some conjectures made by Terwilliger in \cite{terwilligerS3} for general $Q$-polynomial distance-regular graphs. We confirm that these conjectures hold true for cycles. The details are provided below.

\begin{lemma}
	With Assumption \ref{assump: assumptions}, we write $\Lambda = \Lambda_a$. Then the following hold for $h,i,j \in \{0,1,\ldots,D\}$:
	\begin{align}
	E_h^{*(1)}E_i^{*(2)}E_j^{*(3)}\Lambda = \mathbf{0} \text{ if and only if } p^{h}_{ij} = 0, \label{eqn: Eh*1 Ei*2 Ej*3 on Lambda}\\
	E_h^{(1)}E_i^{(2)}E_j^{(3)}\Lambda = \mathbf{0} \text{ if and only if } q^{h}_{ij} = 0. \label{eqn: Eh1 Ei2 Ej3 on Lambda}
	\end{align}
\end{lemma}

\begin{proof}
	Note that 
	\begin{align*}
	\Lambda & = \operatorname{span}\{P_{r,s,t}\ |\ r,s,t \in \{0,1,\ldots,D\} \},\\
	\Lambda & = \operatorname{span}\{Q_{r,s,t}\ |\ r,s,t \in \{0,1,\ldots,D\} \}.
	\end{align*}
	By Definitions \ref{defn: the dual primitive idempotent maps on tensor V} and \ref{defn: the vectors P h i j},  $E_h^{*(1)}E_i^{*(2)}E_j^{*(3)}\Lambda = \operatorname{span}\{P_{h,i,j}\}$. By Definitions \ref{defn: the primitive idempotent maps on tensor V} and \ref{defn: the vectors Q h i j} and (C5), $E_h^{(1)}E_i^{(2)}E_j^{(3)}\Lambda = \operatorname{span}\{Q_{h,i,j}\}$. Now, \eqref{eqn: Eh*1 Ei*2 Ej*3 on Lambda}--\eqref{eqn: Eh1 Ei2 Ej3 on Lambda} hold by (I4) and (J4).
\end{proof}

\begin{remark}\label{remk: first conjecture of Terwilliger}\rm 
	It is known that \eqref{eqn: Eh*1 Ei*2 Ej*3 on Lambda} holds for any $Q$-polynomial distance-regular graph (see \cite[Section 12]{terwilligerS3}). But the claim \eqref{eqn: Eh1 Ei2 Ej3 on Lambda} holds for any $Q$-polynomial distance-regular graph remains open (see \cite[Conjecture 12.2]{terwilligerS3}). 
\end{remark}

\begin{lemma}
		With Assumption \ref{assump: assumptions}, we write $\Lambda = \Lambda_a$. Then the following hold.
		\begin{align}
		\{P_{0,i,i}\ |\ i \in \{0,1,\ldots,D\} \} & \text{ is a basis of }E_0^{*(1)}\Lambda, \label{eqn: E0*1 Lambda}\\
		\{P_{i,0,i}\ |\ i \in \{0,1,\ldots,D\}  \} & \text{ is a basis of } E_0^{*(2)}\Lambda, \\
		\{P_{i,i,0}\ |\ i \in \{0,1,\ldots,D\}  \} & \text{ is a basis of }E_0^{*(3)}\Lambda. \label{eqn: E0*3 Lambda}
		\end{align}
\end{lemma}

\begin{proof}
	We prove \eqref{eqn: E0*1 Lambda}. By Definitions \ref{defn: the dual primitive idempotent maps on tensor V} and \ref{defn: the vectors P h i j} and Corollary \ref{cor: nonzero p h i j}, we have
	\begin{align*}
	\{P_{0,i,j} \neq \mathbf{0}\ |\ i,j \in \{0,1,\ldots,D\} \} & \text{ is a basis of }E_0^{*(1)}\Lambda.
	\end{align*}
	Then \eqref{eqn: E0*1 Lambda} holds by (I4) and (A1)--(A2). The rest are proven similarly.
\end{proof}

\begin{lemma}
		With Assumption \ref{assump: assumptions}, we write $\Lambda = \Lambda_a$. Then the following hold.
	\begin{align}
	\{Q_{0,i,i}\ |\ i \in \{0,1,\ldots,D\} \} & \text{ is a basis of }E_0^{(1)}\Lambda, \label{eqn: E01 Lambda}\\
	\{Q_{i,0,i}\ |\ i \in \{0,1,\ldots,D\}  \} & \text{ is a basis of }E_0^{(2)}\Lambda, \\
	\{Q_{i,i,0}\ |\ i \in \{0,1,\ldots,D\}  \} & \text{ is a basis of }E_0^{(3)}\Lambda. \label{eqn: E03 Lambda}
	\end{align}
\end{lemma}

\begin{proof}
	We prove \eqref{eqn: E01 Lambda}. By (C5), Definitions \ref{defn: the primitive idempotent maps on tensor V} and \ref{defn: the vectors Q h i j}, and Corollary \ref{cor: another basis of Lambda a}, we have
	\begin{align*}
	\{Q_{0,i,j} \neq \mathbf{0}\ |\ i,j \in \{0,1,\ldots,D\} \} & \text{ is a basis of }E_0^{(1)}\Lambda.
	\end{align*}
	Then \eqref{eqn: E01 Lambda} holds by (J4) and (D1)--(D3). The rest are proven similarly.
\end{proof}

\begin{remark}\label{remk: second conjecture of Terwilliger}\rm 
	The claim \eqref{eqn: E0*1 Lambda}--\eqref{eqn: E03 Lambda} hold for any $Q$-polynomial distance-regular graph remains open (see \cite[Conjecture 12.3]{terwilligerS3}).
\end{remark}

\begin{remark}\label{remk: Problem 12.4}\rm 
	It is known that $Q$-polynomial distance regular graphs are equivalent to $P$- and $Q$-polynomial association schemes. Let $\Gamma = (X,R)$ be a $P$- and $Q$-polynomial association scheme. If $\operatorname{Aut}(\Gamma)$ has an abelian subgroup $G'$ that acts regularly on $X$, we say $\Gamma$ is a translation association scheme.   With Assumption \ref{assump: assumptions}, observe that $\Gamma_a$ is an example of $P$- and $Q$-polynomial translation schemes. In \cite{liangtantanakawang}, it was shown that \eqref{eqn: Eh1 Ei2 Ej3 on Lambda} holds for all $P$- and $Q$-polynomial translation schemes.
	It is interesting to know whether \eqref{eqn: E0*1 Lambda}--\eqref{eqn: E03 Lambda} hold for all such schemes (see \cite[Problem 12.4]{terwilligerS3}). 
\end{remark}

\begin{remark}\label{remk: S3-symmetric tridiagonal algebra and sl4}\rm 
	Recently, Martin and Terwilliger \cite{martinterwilliger} established some connections between the $S_3$-symmetric tridiagonal algebra $\mathbb{T}(2,0,0,4,4)$ and the special linear Lie algebra $\mathfrak{sl}_4(\mathbb{C})$. In particular, they proved that the universal enveloping algebra $U(\mathfrak{sl}_4(\mathbb{C}))$ is a homomorphic image of $\mathbb{T}(2,0,0,4,4)$ (see \cite[Definition 3.5, Lemma 3.6, and Note 3.7]{martinterwilliger}). Let $\beta, \gamma, \gamma^*, \varrho , \varrho^*$ denote the scalars in (L3). As a direction for future research, we can investigate which algebraic structure serves as a homomorphic image of the $S_3$-symmetric tridiagonal algebra $\mathbb{T}(\beta,\gamma,\gamma^*,\varrho,\varrho^*)$.
\end{remark}

\section*{Acknowledgment}

	The author gratefully acknowledges Hajime Tanaka for his insightful comments and valuable suggestions, as well as for drawing \cite{terwilligerS3} and \cite{martinterwilliger} to the attention of the author. JVSM also wishes to thank the Graduate School of Information Sciences at Tohoku University for its warm hospitality during the period in which this paper was written. This work was carried out while the author was a Tosio Kato Fellow, and JVSM expresses sincere gratitude to the Mathematical Society of Japan for the support.

\end{document}